\documentclass[11pt,thmsa]{article}%
\usepackage{amsmath}
\usepackage{amsfonts}
\usepackage{amssymb}
\usepackage{graphicx}%
\newtheorem{theorem}{Theorem}[section]

\newtheorem{assertion}[theorem]{Assertion}
\newtheorem{definition}[theorem]{Definition}

\newtheorem{lemma}[theorem]{Lemma}

\newenvironment{proof}[1][Proof]{\noindent\textbf{#1.} }{\ \rule{0.5em}{0.5em}}

\date{}
\begin{document}
\title{Examples of flag-wise positively curved spaces}

\author{Ming Xu\thanks{College of Mathematics, Tianjin Normal University, Tianjin 300387, P. R. China. Email:mgmgmgxu@163.com. Supported by NSFC (no. 11271216), Science and Technology Development Fund for Universities and Colleges in Tianjin (no. 20141005), and Doctor fund of Tianjin Normal University (no. 52XB1305).}}

%\date{March 12th, 2016}
\maketitle

\begin{abstract}
A Finsler space $(M,F)$ is called flag-wise positively curved, if for any
$x\in M$ and any tangent plane $\mathbf{P}\subset T_xM$, we can find
a nonzero vector $y\in \mathbf{P}$, such that the flag curvature $K^F(x,y,
\mathbf{P})>0$. Though compact positively curved spaces are very rare in both Riemannian and Finsler geometry, flag-wise positively curved metrics should be easy to be found. A generic Finslerian perturbation for a non-negatively curved homogeneous metric may have a big chance to produce flag-wise positively curved metrics. This observation leads our discovery of these metrics on many compact manifolds. First we prove any Lie group $G$ such
that its Lie algebra $\mathfrak{g}$ is compact non-Abelian and
$\dim\mathfrak{c}(\mathfrak{g})\leq 1$ admits
flag-wise positively curved left invariant Finsler metrics. Similar techniques can be applied to our exploration for more general compact coset spaces. We will prove, whenever $G/H$ is a compact simply connected coset space, $G/H$ and $S^1\times G/H$ admit flag-wise positively curved Finsler metrics. This provides abundant examples for this type of metrics, which are not homogeneous in general.


\end{abstract}

%\tableofcontents
\section{Introduction}
A {\it Finsler metric} on a smooth manifold $M$ is a continuous function $F:TM\rightarrow [0,+\infty)$ satisfying the following conditions:
\begin{description}
\item{\rm (1)} $F$ is a positive smooth function on the slit tangent bundle
$TM\backslash 0$;
\item{\rm (2)} $F(x,\lambda y)=\lambda F(x,y)$ for any $x\in M$, $y\in T_xM$, and $\lambda\geq 0$;
\item{\rm (3)} For any {\it standard local coordinates} $x=(x^i)$ and $y=y^i\partial_{x^i}$ on $TM$, the Hessian matrix
    $$(g^F_{ij}(x,y))=(\frac12[F^2(x,y)]_{y^iy^j})$$
    is positive definite for any nonzero $y\in T_xM$, i.e. it defines
    an inner product
    $$\langle u,v\rangle_y^F=\frac{1}{2}\frac{d^2}{dsdt}F^2(y+su+tv)|_{s=t=0}
    =g^F_{ij}(x,y)u^iv^j$$
    for any $u=u^i\partial_{x^i}$ and $v=v^j\partial_{x^j}$ in $T_xM$.
\end{description}
We call $(M,F)$ a {\it Finsler space} or a {\it Finsler manifold}. The restriction of the
Finsler metric to a tangent space is called a {\it Minkowski norm}. Minkowski norm can also be defined on any real vector space by similar conditions as (1)-(3), see \cite{BCS00} and \cite{CS05}.

In Finsler geometry, flag curvature is the natural generalization for  sectional curvature in Riemannian geometry. But the flag curvature $K^F(x,y,\mathbf{P})$ is a much more localized geometric quantity in the sense that it depends on tangent plane
$\mathbf{P}\in T_xM$ as well as the nonzero base vector $y\in\mathbf{P}$, see Section 2 below. This inspires us to define the following generalization for
the positively curved condition in Finsler geometry \cite{XD16}.
\begin{definition} Let $(M,F)$ be a Finsler space. We say a tangent plane $\mathbf{P}\subset T_xM$ satisfies the (FP) condition if
there exists a nonzero
vector $y\in T_xM$ such that the flag curvature $K^F(x,y,\mathbf{P})>0$. We
say $(M,F)$ satisfies the (FP) condition or it is flag-wise positively curved
if all its tangent planes satisfy the (FP) condition.
\end{definition}

In \cite{XD16}, we have found many compact coset spaces which admit non-negatively and flag-wise positively curved homogeneous Finsler metrics,
but no positively curved homogeneous Finsler metrics.
If concerning the flag-wise positively curved condition alone, we will have much more chance finding new metrics of this type. For example we can start with a canonical homogeneous metric of non-negative curvature, for example,
bi-invariant metrics on quasi-compact Lie groups (i.e. its Lie algebra is compact), and normal homogeneous metrics \cite{Be}. Then a generic Finslerian perturbation may produce a flag-wise positively curved Finsler spaces. 

In this paper, we will justify this observation. First we will prove the following main theorem, which gives a positive answer to Problem 4.4 in \cite{XD16}.
\begin{theorem}\label{mainthm-1}
Any Lie group $G$ such that $\mathrm{Lie}(G)=\mathfrak{g}$ is a compact
non-Abelian Lie algebra with $\dim\mathfrak{c}(\mathfrak{g})\leq 1$ admits a flag-wise positively curved left invariant Finsler metric.
\end{theorem}

As in Section 4 of \cite{XD16}, where we prove Theorem \ref{mainthm-1} when
$\mathrm{rk}\mathfrak{g}=2$,
the construction for the metric is based on the Killing navigation
technique, but we need a more complicated gluing process here.

With the similar method, we can even prove
\begin{theorem}\label{mainthm-2}
For any compact simply connected coset space $G/H$,
we can find flag-wise positively curved Finsler metrics on
$G/H$ and $S^1\times G/H$.
\end{theorem}

This theorem provides abundant examples of flag-wise positively curved metrics. Notice most metrics in these examples are not homogeneous.

In Section 2, we will briefly summarize some fundamental knowledge on the flag curvature and the Killing navigation technique. In Section 3, we will prove
Theorem \ref{mainthm-1}. In Section 4, we will prove Theorem \ref{mainthm-2}.

\section{Flag curvature and Killing navigation process}

On a Finsler space $(M,F)$, the Riemann curvature
$R_y^F=R_k^i(y)\partial_{x^i}\otimes dx^k: T_x M\rightarrow T_xM$ can be similarly defined as in
Riemann geometry, either by the structure equation of the Chern connection,
or the Jacobi field equation for the variation of geodesics \cite{Shen}.  Using it,
 the flag curvature can be defined as follows. Let $y\in T_xM$
be a nonzero tangent vector ({\it the flag pole}), $\mathbf{P}$ a tangent plane in $T_xM$
containing $y$ ({\it the flag}), and suppose $\mathbf{P}$ is linearly spanned by $y$ and $v$. Then the flag
curvature of the triple $(x,y,y
\wedge v)$ or $(x,y,\mathbf{P})$ is defined as
\begin{equation*}
K^F(x,y,y\wedge v)=
\frac{\langle R_y v,v\rangle^F_y}{\langle y,y\rangle^F_y\langle v,v\rangle^F_y
-(\langle y,v\rangle^F_y)^2}.
\end{equation*}

In fact, the flag curvature $K^F(x,y,y\wedge v)$ is irrelevant to the choice of $v$, so we also denote it as $K^F(x,y,\mathbf{P})$.
When $F$ is a Riemannian metric, it is just the sectional curvature and   irrelevant to the choice of $y$.

The {\it navigation process} is an important technique in studying Randers spaces and flag curvature \cite{BRS04}. Let $V$ be a vector field on the Finsler space $(M,F)$ with
$F(V(x))<1$ for any $x\in M$. Given any $y\in T_xM$, denote
$\tilde{y}=y+F(x,y)V(x)$. Then $\tilde{F}(x,\tilde{y})=F(x,y)$ defines
a new Finsler metric on $M$. We call it the metric defined by the navigation process, or by the {\it navigation
datum} $(F,V)$. When $V$ is a Killing vector field of $(M,F)$, i.e.,
$L_VF=0$, we call this a {\it Killing navigation process}, and $(F,V)$ a {\it Killing navigation datum}. Killing navigation is related to the flag curvature by the following theorem.
\begin{theorem}\label{killingnavigationthm}
Let $\tilde{F}$ be the metric defined by the Killing navigation datum
$(F,V)$ on the smooth manifold $M$ with $\dim M>1$. Then for any $x\in M$, and any nonzero vectors $v$ and $y$ in $T_xM$ such that $\langle v,y\rangle^F_y=0$, we have $K^F(x,y,y\wedge v)=K^{\tilde{F}}(x,\tilde{y}, \tilde{y}\wedge v)$.
\end{theorem}
The proof can be found in \cite{HM07} or \cite{HM15}, where  some more general situations are also considered.

Notice the condition $\langle w,y\rangle^F_y=0$ in Theorem \ref{killingnavigationthm} is equivalent to
$\langle w,\tilde{y}\rangle^{\tilde{F}}_{\tilde{y}}=0$, and the map from $\tilde{y}$ back to $y$ corresponds to the Killing navigation process which
defines $F$ from $(\tilde{F},-V)$.

\section{The proof of Theorem \ref{mainthm-1}}

First we consider the case that the Abelian factor $\mathfrak{g}_0$ in the non-Abelian compact Lie algebra $\mathfrak{g}=\mathrm{Lie}(G)$ is one dimensional.

We start with a
bi-invariant Riemannian metric $F$ on $G$, determined by the bi-invariant
inner product $\langle\cdot,\cdot\rangle_{\mathrm{bi}}$ and the bi-invariant norm $||\cdot||_{\mathrm{bi}}=\langle\cdot,\cdot\rangle_{\mathrm{bi}}^{1/2}$ on $\mathfrak{g}$.

First we consider a Cartan subalgebra $\mathfrak{t}$ of $\mathfrak{g}$ and $v$ a generic vector in $\mathfrak{t}$, i.e.
$\mathfrak{t}=\mathfrak{c}_\mathfrak{g}(v)$. Then $v$ defines a left invariant Killing vector field $V$ for $(G,F)$. For any sufficiently small
$\epsilon>0$, the navigation datum $(F,\epsilon V)$ defines a Finsler
metric $\tilde{F}_{\epsilon}$. Since both $F$ and $V$ are left invariant,
so is $\tilde{F}_{\epsilon}$. By Theorem \ref{killingnavigationthm},
$(G,\tilde{F}_{\epsilon})$ is non-negatively curved. Then we have
the following analog for Lemma 4.3 in \cite{XD16}, with a similar proof.
\begin{lemma} \label{lemma-1}
(1) Keep all the above assumptions and notations and fix any sufficiently small $\epsilon >0$. If the 2-dimensional subspace
$\mathbf{P}\subset\mathfrak{g}$ does not satisfy the (FP) condition, i.e. $K^{\tilde{F}_{\epsilon}}(e,y,\mathbf{P})\leq 0$ for any nonzero  $y\in\mathbf{P}$, then $\mathbf{P}\subset\mathfrak{t}$.

(2)
When $\mathbf{P}$ is not contained in $\mathfrak{t}$,
$K^{\tilde{F}_{\epsilon}}(e,y,\mathbf{P})> 0$ for any nonzero generic $y\in\mathbf{P}$.
\end{lemma}
\begin{proof}
(1) Given any $\mathbf{P}\subset\mathfrak{g}$ as in the lemma, we can find a nonzero vector $w_2\in\mathbf{P}$ with $\langle w_2,v\rangle_{\mathrm{bi}}=0$. Then there exists
a nonzero vector $w_1\in\mathbf{P}$ such that $\langle w_1,w_2\rangle_{\mathrm{bi}}=0$. We can also find nonzero vectors $w'_1$
and $w'_2$ satisfy
$$\tilde{w'_1}=w'_1+\epsilon F(w'_1)v=w_1\mbox{ and }
\tilde{w'_2}=w'_2+\epsilon F(w'_2)v=-w_1.$$
Moreover, we also have
$$\langle w'_1,w_2\rangle_{\mathrm{bi}}=\langle w'_2,w_2\rangle_{\mathrm{bi}}=0.$$
 By Theorem \ref{killingnavigationthm}, we have
\begin{eqnarray}
K^{F}(e,w'_1\wedge w_2)&=&K^{\tilde{F}_{\epsilon}}(e,w_1,\mathbf{P})\leq 0,\label{0000}
\end{eqnarray}
 and
 \begin{eqnarray}
K^F(e,w'_2\wedge w_2)&=&K^{\tilde{F}_{\epsilon}}(e,-w_1,\mathbf{P})\leq 0.\label{0001}
\end{eqnarray}
 Since $(G,F)$ is non-negatively curved, the  equality holds for both (\ref{0000}) and (\ref{0001}), that is,  we have
$$
[w'_1,w_2]=[w_1,w_2]-\epsilon F(w'_1)[v,w_2]=0,$$
 and
$$[w'_2,w_2]=-[w_1,w_2]-\epsilon F(w'_2)[v,w_2]=0.$$
Because $\epsilon$, $F(w'_1)$ and $F(w'_2)$ are all positive,
we conclude that $[w_1,w_2]=[v,w_2]=0$.
So we have $w_2\in\mathfrak{c}_\mathfrak{g}(v)=\mathfrak{t}$.

Now if we change the flag pole to another generic $w_3=w_1+cw_2\in\mathbf{P}$, $c\neq 0$,
then there is a nonzero number $d$ such that the vector $w_4=w_2+dw_1$ satisfies the condition
$\langle w_3,w_4\rangle^{\tilde{F}_{\epsilon}}_{w_3}=0$. Notice $F$ is also defined by the Killing navigation datum $(\tilde{F}_\epsilon,-V)$. Then by Theorem \ref{killingnavigationthm}, for $w'_3=w_3-\epsilon\tilde{F}_\epsilon(w_3)v$, we have
$$K^F(e,w'_3\wedge w_4)=K^{\tilde{F}_{\epsilon}}(e,w_3,w_3\wedge w_4)\leq 0.$$
So we have $K^F(e,w'_3\wedge w_4)=0$, and
$$[w'_3,w_4]=[w_1+cw_2-\epsilon \tilde{F}_\epsilon(w_3)v,w_2+dw_1]=-d\epsilon \tilde{F}_\epsilon(w_3)[v,w_1]=0.$$
Because $d$, $\epsilon$ and $\tilde{F}_\epsilon(w_3)$ are nonzero numbers, we must have $[v,w_1]=0$, i.e. $w_1\in\mathfrak{c}_\mathfrak{g}(v)=\mathfrak{t}$. Thus $\mathbf{P}=\mathrm{span}\{w_1,w_2\}\subset\mathfrak{t}$.

(2)
When $\mathbf{P}$ is not contained in $\mathfrak{t}$, we have just proved
$K^{\tilde{F}_\epsilon}(e,y,\mathbf{P})>0$ for some nonzero vector
$y\in\mathbf{P}$. Notice the left invariant metric $\tilde{F}_{\epsilon}$ is real analytic. So the same statement must be valid for nonzero generic vectors.
\end{proof}

Denote $\mathcal{S}=\{w\in\mathfrak{g},\quad ||w||_{\mathrm{bi}}=1\}\subset\mathfrak{g}$
the bi-invariant unit sphere in $\mathfrak{g}$. For the one-dimensional Abelian factor
$\mathfrak{g}_0$, we have $\mathcal{S}\cap\mathfrak{g}_0=\{\pm u_0\}$.

For any $u\in\mathcal{S}\backslash\{\pm u_0\}$
we can find
a Cartan subalgebra $\mathfrak{t}$ such that $u\notin\mathfrak{t}$. Then there
exists a sufficiently small $r>0$, such that the open neighborhood
$$\mathcal{U}_{u,r}=\{w\in\mathcal{S},\quad ||w-u||_{\mathrm{bi}}< r\}$$ of $u$ in
$\mathcal{S}$ satisfies $\overline{\mathcal{U}_{u,r}}\cap\mathfrak{t}=\emptyset$ (especially, $\pm u_0\notin\overline{\mathcal{U}_{u,r}}$), and $\overline{\mathcal{U}_{u,r}}$ covers less than half of
$\mathcal{S}$. Notice its boundary in $\mathcal{S}$, $\partial{\mathcal{U}_{u,r}}=\{w\in\mathcal{S},\quad||w-u||_{\mathrm{bi}}=r\}$,
is a co-dimension one sphere with a small radius, and it is the intersection between $\mathcal{S}$ and a hyperplane.

Take any generic $v$ from $\mathfrak{t}$, and any sufficiently small $\epsilon>0$, by Lemma \ref{lemma-1}, the metric $\tilde{F}_\epsilon$ defined above satisfies
\begin{assertion}If the 2-dimensional subspace $\mathbf{P}\subset\mathfrak{g}$ satisfies $\mathbf{P}\cap\mathcal{U}_{u,r}\neq\emptyset$, then
for any nonzero generic vector $y\in\mathbf{P}\cap\mathcal{U}_{u,r}$, we have  $K^{\tilde{F}_\epsilon}(e,y,\mathbf{P})>0$.
\end{assertion}
The intersection of a 2-dimensional subspace $\mathbf{P}$ with $\mathcal{S}$
will called a {\it big circle}.

The open neighborhoods $\mathcal{U}_{u,r}$ for all $u\in\mathcal{S}\backslash\{\pm u_0\}$ provide a open covering for $\mathcal{S}\backslash\{\pm u_0\}$. To make a finite open covering for $\mathcal{S}$, we need two more neighborhoods of $\pm u_0$,
$$\mathcal{U}^\pm=\{w\in\mathcal{S},\quad||\pm u_0-w||<r_0\},$$
where $r_0$ is a sufficiently small positive number. We denote this finite open covering for $\mathcal{S}$ as $\{\mathcal{U}^+,\mathcal{U}^-,\mathcal{D}_{u_i,r_i}, 1\leq i\leq m\}$.
In previous argument, each $\mathcal{D}_{u_i,r_i}$ has been associated with
a left invariant Finsler metric $\tilde{F}_{i;\epsilon}$ by the Killing navigation technique.

%If $\mathfrak{g}$ has a nontrivial Abelian factor $\mathfrak{g}_0$, there are exactly two vectors $\pm u_0\in\mathfrak{g}_0\cap\mathcal{S}$. Fix a sufficiently small $r_0>0$. Then there is a finite covering of $\mathcal{S}$
%by $\mathcal{D}_{u_0,r_0}$, $\mathcal{D}_{-u_0,r_0}$, and a finite sequence
%of $\mathcal{D}_{u_i,r_i}$, $i=1,\ldots,n$, as described above. For each
%$\mathcal{D}_{u_i,r_i}$ with $1\leq i\leq n$, we have the corresponding left
%invariant Finsler metric $\tilde{F}_{i;\epsilon}$ from the Killing navigation process, for some generic vector $v_i$ in a Cartan subalgebra $\mathfrak{t}_i\subset \mathfrak{g}\backslash\{u_i\}$.

Denote $\mathcal{S}'$ the union of the following co-dimension one spheres in $\mathcal{S}$,
$\partial\mathcal{U}^+$,
$\partial\mathcal{U}^-$,
and $\partial\mathcal{U}_{u_i,r_i}$ for all $1\leq i\leq m$. For any $\delta>0$, denote
$$\mathcal{S}'_\delta=\{w\in\mathcal{S},\quad|| w-w',w-w'||_{\mathrm{bi}}\leq\delta\mbox{ for some }w'\in\mathcal{S}'\}.$$
The complement of $\mathcal{S}'_\delta\cup\overline{\mathcal{U}^+}\cup\overline{\mathcal{U}^-}$ in $\mathcal{S}$ for a sufficiently small $\delta>0$ is a disjoint finite union of connected open subsets
$\mathcal{V}_i$ of $\mathcal{S}$, $i=1,\ldots,N$. Notice their closures
$\overline{\mathcal{V}_i}$ are disjoint as well.
If $\mathcal{U}_i$ is contained by some $\mathcal{D}_{u_j,r_j}$, we define
the metric $F_{i;\epsilon}$ to be the corresponding $\tilde{F}_{j;\epsilon}$.
When we have multiple choices of $F_{i;\epsilon}$, just choose any one of them.

The key observation here is the following lemma.
\begin{lemma}\label{lemma-2}
Keep all relevant assumptions and notations above. Then for a sufficiently small $\delta>0$,
$\mathcal{S}\backslash(\mathcal{S}'_\delta
\cup\overline{\mathcal{U}^+}\cup\overline{\mathcal{U}^-})$ has
a nonempty intersection with any big circle (or equivalently, any 2-dimensional subspace $\mathbf{P}$).
\end{lemma}
\begin{proof}
Assume conversely that $\delta$ indicated by the lemma does not exist, then
for any $n\in\mathbb{N}$, there is a big circle $\mathcal{C}_n=\mathcal{S}\cap\mathbf{P}_n\subset\mathcal{S}'_{1/n}
\cup\overline{\mathcal{U}^+}\cup\overline{\mathcal{U}^-}$. Passing to a suitable subsequence, we can get
a limit big circle $\mathcal{C}=\lim\mathcal{C}_n\subset
\mathcal{S}'\cup\overline{\mathcal{U}^+}\cup\overline{\mathcal{U}^-}$. Because
the big circle $\mathcal{C}$ can not be contained by the two small disks $\overline{\mathcal{U}^+}$ and $\overline{\mathcal{U}^-}$, the part of $\mathcal{C}$ covered by $\mathcal{S}'$ must have a positive length.
But $\mathcal{S}'$ is a finite union of co-dimension 1 spheres
with small radii. Each sphere in $\mathcal{S}'$ can only intersect $\mathcal{C}$ at finite points, i.e. $\mathcal{C}\cap\mathcal{S}'$ is a finite set.
This is a contradiction.
\end{proof}

Fix a $\delta>0$ indicated by Lemma \ref{lemma-2}.
Now we are ready to construct the left invariant metric indicated by Theorem
\ref{mainthm-1}. Let the sequence of non-negative smooth functions $\mu_1,\ldots,\mu_N$ on $\mathcal{S}$ be a partition of unit, i.e.
$\sum_{i=1}^N\mu_i\equiv 1$, and $\mu_i|_{\mathcal{V}_j}\equiv \delta_{ij}$.
The smooth functions $\mu_i$ can also be viewed as positively homogeneous functions of degree 0 on $\mathfrak{g}\backslash\{0\}$.
Denote $F_\epsilon=\sum_{i=1}^N\mu_i F_{i;\epsilon}$. Because $F_0$ coincides with the bi-invariant Riemannian norm on $\mathfrak{g}$, $F_\epsilon$ with sufficiently small $\epsilon>0$ satisfies the positive definite condition for
the Hessian of $F_\epsilon$. Fix a sufficiently small $\epsilon>0$, $F_\epsilon$ defines a Minkowski norm on $\mathfrak{g}$, and translations by $G$ defines a left invariant Finsler metric, still denoted as $F_\epsilon$.

Finally we check the (FP) condition for $F_\epsilon$. We only need to prove it
at $e$. For any tangent plane $\mathbf{P}\subset T_eG=\mathfrak{g}$, by Lemma
\ref{lemma-2}, the big circle $\mathbf{P}\cap\mathcal{S}$ will have nonempty
intersection with some $\mathcal{V}_i\subset\mathcal{U}_{u_j,r_j}$. Notice the associated metric $F_{i;\epsilon}=\tilde{F}_{j;\epsilon}\neq F$ is defined by
a Killing navigation process. Then by Lemma \ref{lemma-1}, for
any nonzero generic vector $y\in\mathcal{U}_i\cap\mathbf{P}$,
$K^{F_\epsilon}(e,y,\mathbf{P})=K^{\tilde{F}_{j;\epsilon}}(e,y,\mathbf{P})>0$.

The above argument proves Theorem \ref{mainthm-1} when the Abelian factor of $\mathfrak{g}$ is one-dimensional. When $\mathfrak{g}$ has no Abelian factor,
we can just assume $\mathcal{U}^\pm=\emptyset$, then the same argument All also proves the theorem in this case.

\section{The proof of Theorem \ref{mainthm-2}}
First we consider the case that
$M=S^1\times G/H$ where $G/H$ is a simply connected compact coset space.

%Denote $\pi_i$ the projection of $M$ to each factor $G_i/H_i$, and ${\pi_i}_*$
%its tangent map.
We can assume $G$ is a compact Lie group. Respect to a fixed bi-invariant inner product $\langle\cdot,\cdot\rangle_{\mathrm{bi}}$, we have the orthogonal decomposition $\mathfrak{g}=\mathfrak{h}+\mathfrak{m}$, and a normal homogeneous Riemannian metric $F'$ on $G/H$. Then $F^2=dt^2+F'^2$
defines a normal homogeneous Riemannian metric on $M$.
%The inner product and norm $F$ defines on each tangent space will be simply denoted as $\langle\cdot,\cdot\rangle_F$ and $|\cdot|_F$.
Denote $\mathcal{S}M$ the sphere bundle over $M$, consisting of all $F$-unit
tangent vectors. There are exactly two smooth sections of the bundle $\mathcal{S}M$, corresponding to the $F$-unit tangent vectors from the $S^1$-directions. Denote their imagines as $\mathcal{E}^+$ and $\mathcal{E}^-$ respectively.

Consider any $x=(x_0,x_1)\in M$ with $x_0\in S^1$ and $x_1\in G/H$.
We can suitable choose the presentation of $G/H$ to make $x_1=eH$.
Then a tangent plane
$\mathbf{P}\subset T_xM=\mathbb{R}\oplus\mathfrak{m}$ has a 0 sectional curvature for the metric $F$ iff $\mathbf{P}$ can be
spanned by $u=(t,u_1)$ and $v=(t',v_1)$ with $[u_1,v_1]=0$.

Let $w=(s,w_1)$ be any tangent vector in $\mathcal{S}M_x\backslash\mathcal{E}^\pm$ with $w_1\neq 0$.
Then we have the following lemma.
\begin{lemma} \label{lemma-0}
Keep all above notations and assumptions. Then there exists
a nonzero vector $v_1\in\mathfrak{m}$, such that $[w_1,v_1]\neq0$, and $\langle w_1,v_1\rangle_{\mathrm{bi}}=0$.
\end{lemma}
\begin{proof}
We only need to prove $[w_1,\mathfrak{m}]\neq 0$, then the existence of $v_1$ is obvious.
Assume conversely $[w_1,\mathfrak{m}]=0$, then we also have
$[w_1,[w_1,\mathfrak{h}]]=[w_1,\mathfrak{m}]=0$. This implies $[w_1,\mathfrak{h}]=0$, i.e. $w_1\in\mathfrak{c}(\mathfrak{g})\cap\mathfrak{m}$. The simply connected
$G/H$ must has an Euclidean product factor. This is a contradiction.
\end{proof}

Using $v_1\in\mathfrak{m}$ indicated by lemma \ref{lemma-0}, we can get a Killing vector field $V$ of $(M,F)$ defined by $(0,v_1)$. Because $\langle w_1,v_1\rangle_{\mathrm{bi}}=0$, $V(x)$ is $F$-orthogonal to $w$. For any sufficiently
small $\epsilon>0$, we have a Finsler metric $\tilde{F}_{\epsilon}$ induced by
the navigation datum $(F,\epsilon V)$. By Theorem \ref{killingnavigationthm},
$(M,\tilde{F}_\epsilon)$ is non-negatively curved.

Similar to Lemma \ref{lemma-1}, we have
\begin{lemma}\label{lemma-3}
Keep all relevant assumptions and notations. Fix a sufficiently small $\epsilon>0$.
Then for any tangent plane $\mathbf{P}\subset T_xM$ containing $w$, the flag curvature $K^{\tilde{F}_\epsilon}(x,y,\mathbf{P})>0$ for nonzero
generic vector $y\in\mathbf{P}$.
\end{lemma}

\begin{proof}
Because the metric $\tilde{F}_\epsilon$ is real analytic, we only need to
prove the (FP) condition for $\mathbf{P}$. Assume conversely it is
not true, i.e. for any nonzero $y\in \mathbf{P}$, $K^{\tilde{F}_\epsilon}(x,y,\mathbf{P})\leq 0$. We can find a nonzero
$w'\in\mathbf{P}$ which is $F$-orthogonal to $w$.
Then there are nonzero vectors $v'$ and $v''$ in $T_xM$, such that
$$\tilde{v'}=v'+\epsilon F(v')V(x)=w'\mbox{ and }\tilde{v''}=v''+\epsilon F(v'')V(x)=-w'.$$
Since our assumption implies that $w$ be $F$-orthogonal to $V(x)$, so does $w$ to $v'$ and $v''$. By Theorem \ref{killingnavigationthm}, we have
$$K^F(x,v'\wedge w)=K^{\tilde{F}_\epsilon}(x,w',w\wedge w')=K^{\tilde{F}_\epsilon}(x,w',\mathbf{P})\leq 0$$
and
$$K^F(x,v''\wedge w)=K^{\tilde{F}_\epsilon}(x,-w',w\wedge w')=K^{\tilde{F}_\epsilon}(x,w',\mathbf{P})\leq 0.$$
Because $(M,F)$ is non-negatively curved, we have $K^F(x,v'\wedge w)=K^F(x,v''\wedge w)=0$. Denote the $\mathfrak{g}$-factors of $w'$, $v'$ and $v''$ as $w'_1$, $v'_1$ and $v''_1$ respectively, then both $v'_1=w'_1-\epsilon F(v')v_1$ and $v''_1=-w'_1-\epsilon F(v'')v_1$ commute with $w_1$. Because $\epsilon$, $F(v')$ and $F(v'')$ are positive numbers, we get $[w_1,v_1]=0$. This is a contradiction.
\end{proof}

The property of $w$ in Lemma \ref{lemma-3} can also be passed to other tangent vectors in $\mathcal{S}M$ which are sufficiently closed to $w$. To be precise,
we have the following lemma.

\begin{lemma} \label{lemma-4}
Keep all relevant assumptions and notations above. Fix a sufficiently small $\epsilon>0$. Then there exist a sufficiently small neighborhood $\mathcal{U}$ of $w$ in $\mathcal{S}M$
satisfying the following property, if a tangent plane $\mathbf{P}'\subset T_{x'}M$ has non-empty intersection with $\mathcal{U}$,
$K^{\tilde{F}_\epsilon}(x',y',\mathbf{P}')>0$ for nonzero generic $y'\in\mathbf{P}'\cap\mathcal{U}$.
\end{lemma}

\begin{proof}
Assume conversely that there does not exist such a neighborhood $\mathcal{U}$. Then there exist a sequence of tangent planes $\mathbf{P}_n\subset T_{x_n}M$, and tangent vectors $w_n\in\mathcal{S}M\cap\mathbf{P}_n$, such that $\lim x_n=x$, $\lim w_n=w$, and $K^{\tilde{F}_\epsilon}(x_n,y,\mathbf{P}_n)=0$ for each $n$ and each nonzero $y\in\mathbf{P}_n$. Passing to a suitable subsequence, $\mathbf{P}_n$ converge to is a tangent plane $\mathbf{P}\subset T_xM$
containing $w$. Then by continuity,
$K^{\tilde{F}_\epsilon}(x,y,\mathbf{P})=0$ for each nonzero vector $y\in\mathbf{P}$. This is a contradiction to Lemma \ref{lemma-3}.
\end{proof}

Whenever we have found a neighborhood $\mathcal{U}$ of $w$ in $\mathcal{S}M$
indicated by Lemma \ref{lemma-4}, any smaller neighborhood of $w$ also satisfies the same property. Because $w$ is not contained in $\mathcal{E}^\pm$,
we can also assume $\overline{\mathcal{U}}\cap\mathcal{E}^\pm=\emptyset$.

We further require $\mathcal{U}$ to have the following presentation. Take a
sufficiently small closed neighborhood $\mathcal{B}\subset M$ of $x$, and  a smooth
local section $s(\cdot):\mathcal{B}\rightarrow \mathcal{S}M$ with $s(x)=w$. Next we choose a smooth function $r(\cdot):\mathcal{B}\rightarrow[0,+\infty)$ such
that $r\equiv 0$ on $\partial\mathcal{B}$ and $r>0$ sufficiently small inside $\mathcal{B}$.
Then
$$\mathcal{U}=\{u'\in\mathcal{S}_{x'}M,\quad x'\in\mathcal{B}, F(u'-s(x'))<r(x')\}$$
is a sufficiently small neighborhood hood $w$ in $\mathcal{S}M$. Denote
$\partial\mathcal{U}$ its boundary in $\mathcal{S}M$.
For $x'\in M$ inside $\mathcal{B}$, the intersection $\partial\mathcal{U}\cap\mathcal{S}_{x'}M$ is a co-dimension 1 sphere
$\{u'\in\mathcal{S}_{x'}M,\quad F(u'-s(x'))=r(x')\}$ in $\mathcal{S}_{x'}M$, which
is the intersection of $\mathcal{S}_{x'}M$ with some hyperplane. For other $x'$, $\partial\mathcal{U}\cap\mathcal{S}_{x'}M$ is an empty set or just a point.

To summarize, the neighborhoods $\mathcal{U}$ constructed above for all $w\in\mathcal{S}M\backslash(\mathcal{E}^+\cup\mathcal{E}^-)$ provide an open covering for $\mathcal{S}M\backslash(\mathcal{E}^+\cup\mathcal{E}^-)$. To get a finite open covering for $\mathcal{S}M$, we just need to add the following two open neighborhoods of $\mathcal{E}^\pm$,
$$\mathcal{U}^\pm=\bigcup_{x\in M}\{w\in\mathcal{S}_xM,\quad F(w-w')<{r_0}\mbox{ for some }w'\in S_xM\cap\mathcal{E}^\pm\},$$
where the fixed positive number $r_0$ is sufficiently small. Denote the open covering of $\mathcal{S}M$ as $\{\mathcal{U}^+,\mathcal{U}^-,\mathcal{U}_1, \ldots, \mathcal{U}_m\}$. In previous argument, each $\mathcal{U}_i$ is associated with the Finsler metrics $\tilde{F}_{i;\epsilon}$.

Denote $\mathcal{S}'$ the union of all boundaries $\partial\mathcal{U}^\pm$ and $\partial\mathcal{U}_i$ in $\mathcal{S}$, and for any $\delta>0$,
$$\mathcal{S}'_\delta=\bigcup_{x\in M}\{u\in\mathcal{S}_xM,\quad F(u-w)\leq\delta \mbox{ for some }w\in\mathcal{S}'\cap T_xM\}.$$
Similar to Lemma \ref{lemma-2}, we have the following
\begin{lemma} \label{lemma-5}
Keep all relevant assumptions and notations above. Then for a
sufficiently small $\delta>0$, any tangent plane must have a non-empty intersection with $\mathcal{S}M\backslash(\mathcal{S}'_\delta\cup
\overline{\mathcal{U}_0^+}\cup\overline{\mathcal{U}_0^-})$
\end{lemma}
\begin{proof}
Assume the number $\delta$ indicated by the lemma does not
exist. Then we can find a sequence $x_n\in M$, and a sequence of
tangent planes $\mathbf{P}_n\subset T_{x_n}M$ such that the big circle $\mathcal{C}_n=\mathbf{P}_n\cap\mathcal{S}_xM\subset\mathcal{S}'_{1/n}\cup
\overline{\mathcal{U}^+}\cup\overline{\mathcal{U}^-}$. Passing to a suitable limit, we will
have $\lim x_n=x$ and $\lim\mathcal{C}_n=\mathcal{C}$ which is a big circle (i.e. the intersection between a tangent plane $\mathbf{P}\subset T_xM$ and
$\mathcal{S}_xM$) contained in
$\mathcal{S}_xM\cap(\mathcal{S}'\cup\overline{\mathcal{U}_0^+}\cup\overline{\mathcal{U}_0^-})$. Because $\mathcal{C}$ can not be contained in the two small disks $\mathcal{S}_xM\cap\overline{\mathcal{U}_0^\pm}$, so the part
of $\mathcal{C}$ contained in $\mathcal{S}'$ must have a positive length. But
$\mathcal{S}'$ is a finite union of co-dimension one spheres of small radii,
which are intersections of $\mathcal{S}_xM$ with hyperplanes. So the intersection between $\mathcal{C}$ and $\mathcal{S}'$ is a finite set. This is
a contradiction.
\end{proof}

Fixed a sufficiently small $\delta>0$ indicated by Lemma \ref{lemma-5}.
The complement $\mathcal{S}M\backslash(\mathcal{S}'_\delta\cup\overline{\mathcal{U}^+}
\cup\overline{\mathcal{U}^-})$ is a disjoint union of connected open subsets in $\mathcal{S}M$. To see it is a finite union, we first observe only finite open components of $\mathcal{S}M\backslash(\mathcal{S}'_\delta\cup\overline{\mathcal{U}^+}
\cup\overline{\mathcal{U}^-})$
intersect each $S_xM$, and then use the finite open covering technique for the compact manifold $M$.
Denote these disjoint open subsets of $\mathcal{S}$ as $\mathcal{V}_1$, $\ldots$, $\mathcal{V}_N$. Their closures
$\overline{\mathcal{V}_i}$ are disjoint as well. Each $\mathcal{V}_i$ is contained by some $\mathcal{U}_j$, which in previous discussion is associated with Finsler metrics $\tilde{F}_{j;\epsilon}$ by the Killing navigation process, we then define
$F_{i;\epsilon}=\tilde{F}_{j;\epsilon}$ associated with $\mathcal{V}_i$. If we have multiple choices for $F_{i;\epsilon}$, just choose any one.

Let the non-negative smooth functions $\mu_1$, $\ldots$, $\mu_N$ on $\mathcal{S}M$ be a partition of unit, i.e. $\sum_{i=1}^N\mu_i\equiv 1$, such that
$\mu_i|_{\mathcal{V}_j}\equiv\delta_{ij}$. They will also be viewed as positively
homogeneous functions of degree 0 on the slit tangent bundle $TM\backslash 0$.

Now we are ready to construct the Finsler metric indicated by Theorem \ref{mainthm-2}, $F_\epsilon=\sum_{i=1}^N \mu_i F_{i;\epsilon}$. When $\epsilon=0$, we have $F_0=F$. So fix any sufficiently small $\epsilon>0$,
$F_\epsilon$ satisfies positive definite condition for its Hessian, and thus
$F_\epsilon$ is a Finsler metric on $M$.

Finally we check the (FP) condition for $F_\epsilon$. Consider any tangent plane $\mathbf{P}\subset T_xM$. By lemma \ref{lemma-5},
$\mathbf{P}\cap\mathcal{V}_i\neq\emptyset$ for some $i$. On $\mathcal{V}_i$, $F_\epsilon$ coincides with
 some $F_{i;\epsilon}=\tilde{F}_{j;\epsilon}$ with $\mathcal{V}_i\subset\mathcal{U}_j$.
 By Lemma \ref{lemma-4}, for nonzero generic $y\in\mathbf{P}\cap\mathcal{V}_i$,
$$K^{F_\epsilon}(x,y,\mathbf{P})=K^{\tilde{F}_{i;\epsilon}}(x,y,\mathbf{P})>0,$$
i.e. the (FP) condition is satisfied for $(M,F_\epsilon)$.

This proves
Theorem \ref{mainthm-2} when $M$ has an $S^1$ product factor.
When $M$ does not have the $S^1$ product factor, we can simply assume
$\mathcal{E}^{\pm}=\mathcal{U}_0^\pm=\emptyset$, then the above argument also
proves Theorem \ref{mainthm-2} in this case.

\end{document}